\documentclass{article}

\def\cU{\mathcal{U}}
\def\cP{\mathcal{P}}
\def\cE{\mathcal{E}}
\def\cC{\mathcal{C}}

\usepackage{amsmath,amssymb,amsthm,graphicx,verbatim,tabularx}
\usepackage{algorithmic,algorithm,newfloat,color}
\usepackage[hyphens]{url}

\newtheorem{theorem}{Theorem}

\newtheorem{corollary}{Corollary}

\DeclareFloatingEnvironment[name=Algorithm]{alg}

\begin{document}

\title{Classifying with Uncertain Data Envelopment Analysis}

\author{Casey Garner$^{a\dagger}$ and Allen Holder$^{\, b^*\dagger}$ \\ \\
\parbox{.9\textwidth}{\footnotesize
$^a\,$Department of Mathematics, University of Minnesota,
	Minneapolis, MN, USA, (garne214@umn.edu)} \\[8pt]
\parbox{.9\textwidth}{\footnotesize
$^b\,$Department of Mathematics, Rose-Hulman Institute of Technology,
        Terre Haute, IN, USA (holder@rose-hulman.edu)} \\[4pt]
\parbox{0.9\textwidth}{\footnotesize
$^*$Corresponding author} \\[0pt]
\parbox{.9\textwidth}{\footnotesize
$^\dagger\,$ authors listed alphabetically and contributed equally,
	research conducted at Rose-Hulman Institute of Technology} \\[8pt]
}

\maketitle

\begin{abstract}
Classifications organize entities into categories that identify
similarities within a category and discern dissimilarities among 
categories, and they powerfully classify information in support 
of analysis. We propose a new classification scheme premised 
on the reality of imperfect data. Our computational model 
uses uncertain data envelopment analysis to define a 
classification's proximity to equitable efficiency, which 
is an aggregate measure of intra-similarity within a classification's 
categories. Our classification process has two overriding 
computational challenges, those being a loss of convexity and a
combinatorially explosive search space.  We overcome the first by
establishing lower and upper bounds on the proximity value, and
then by searching this range with a first-order algorithm.
We overcome the second by adapting the p-median problem to initiate
our exploration, and by then employing an iterative
neighborhood search to finalize a classification. We
conclude by classifying the thirty stocks in the Dow 
Jones Industrial average into performant tiers and by classifying 
prostate treatments into clinically effectual categories. \\\

\noindent{\bf Keywords:} Data Envelopment Analysis, 
    Robust Optimization, \\[1pt]
\hspace*{54pt}Classification, Clustering \\[4pt]
\end{abstract}

\section{Motivation and Introduction} \label{sec-intro}

Classifications are ubiquitous and essential constructs across the human
experience, a reality motivated by the fact that  we learn through association 
and disassociation.  Indeed, our thoughts are replete with comparisons and
contrasts among known and novel aspects of our understanding, and we routinely
toil to organize and group information into like collections for myriad
purposes.  This overriding narrative to categorize information manifests itself
in the joys of life, say by classifying literature, art, food, and sport; it
promotes scientific progress by delineating the fields of study and
by interpreting the results of experimentation; and it aids innovation in
fields like engineering, healthcare, and politics as we study how the practical
benefits of basic research can promote the general welfare.

The classification concept has not escaped mathematical and practical 
study, and the process of dividing entities into categories has a long 
and fruitful history. The literature on how to classify is consequently 
substantial, well studied, broad in application, and methodologically 
diverse as illustrated by the following compendious list. \\
\begin{center}
\renewcommand{\tabcolsep}{12pt}
\renewcommand{\arraystretch}{1.2}
\begin{tabular}{ll}
\hline \\[-8pt]
decision trees~\cite{vens2008} &
	integer programming~\cite{bertsimas2007} \\
neural networks~\cite{zhang2000} &
	robust statistics~\cite{garcia2010} \\
Bayesian networks~\cite{phyu2009} &
	discriminant analysis~\cite{huberty2006} \\
graph clustering~\cite{schaeffer2007} &
data mining~\cite{olafsson2008} \\[4pt]
\hline
\end{tabular} \\[\baselineskip]
\end{center} 
Additional classification techniques, along with numerous citations, 
are reviewed in~\cite{duran2013,jain1999,kotsiantis2007,saxena2017}.
This disciplinary range challenges having a categorical language that
would promote shared efforts across disciplines, and while all of mathematics,
computer science, statistics, probability, data science, operations
research, and machine learning dabble with similar classification problems,
they each scope their problems in terms apropos for their specific 
tasks, see~\cite{jain1999,lee2010} for related comments.  We use the
term classification diffusely to simply mean that entities are
partitioned into categories, and we refrain from related terms like 
clustering and grouping that have more refined definitions in some 
settings.

We propose a classification technique motivated by uncertain data envelopment 
analysis (uDEA) that scores a classification by how intra-equitable its 
categories are. Our fundamental assumptions are that:
\begin{enumerate}
\item classifications of the same entities are to be compared on a \\ 
    common set of characteristics, and
\item the characteristics are imperfect.
\end{enumerate}
The assumption of imperfect characteristics is surely an 
anticipated general rule and not a rare oddity; after all, a lack of
preciseness taints everything from scientific experimentation, to expert 
opinion, and to rule based delineation. We do not make any 
stochastic assumption, and hence, we bypass all hassle of verifying 
distributional suppositions.  That said, we can interpret our model 
stochastically if such an interpretation appropriately aids analysis.  We 
instead use the modern concept of uncertainty associated with robust 
optimization to account for imperfect information, and our model builds 
from the work in~\cite{ehrgott2018} to leverage uncertainty toward the 
goal of establishing equity within a classification's categories. We comment 
that data envelopment analysis (DEA) is not new to the processes of clustering
and classification, see~\cite{cinaroglu2020,po2009,thanassoulis1996} as
illustrative examples.

Sections~\ref{sec-eqEff} and~\ref{sec-algorithm} define the proximity 
to equitable efficiency and give an algorithm to approximate its value.
We present our classification scheme in Section~\ref{sec-classification}, and
we review examples in finance and medicine in Section~\ref{sec-examples}.
Note that all norms are $2$-norms unless otherwise noted, although
we explicitly denote all norms in the proof of Theorem~\ref{thm-capability} 
to avoid confusion.

\section{Proximity to Equitable Efficiency} \label{sec-eqEff}

We define an assessment metric that models the quality of a classification. 
Assume $\mathcal{O} = \left \{ o^t : t=1,2,... T \right \}$ is a finite 
set of objects that are to be partitioned into the nonempty disjoint 
categories $C^s$, for $s=1,2,..., S$. So each $o^t$ is in some $C^s$, 
and $\cC = \{C^s: s=1,2,...,S \}$ partitions $\mathcal{O}$. 
We assess a classification $\cC$ by defining a measure of the
intra-similarity of the elements within the individual categories 
$C^s$, and the classification's overall quality then equates to the
sum of its categorical values. 

Characteristics describe individual objects, and we assume that a
classification's quality coincides with some of these 
characteristics being small and some of them being large. Our 
language here agrees with the DEA literature, 
see~\cite{cooper2007,emrouznejad2008,hwang2016,liu2013} as
reviews, and characteristics whose decrease signifies improvement are called
input-characteristics, or more simply inputs, and those whose 
increase signifies improvement are called output-characteristics, 
or more simply outputs.  We let 
\[
\renewcommand{\arraycolsep}{1pt}
\begin{array}{rl}
\displaystyle
X^s_{it} & \mbox{be the value of the $i$-th input-characteristic of 
	object $t$ in category $C^s$, and} \\[4pt]
\displaystyle
Y^s_{it} & \mbox{be the value of the $i$-th output-characteristic of 
	object $t$ in category $C^s$.}
\end{array}
\]
We assume there are $N$ inputs and $M$ outputs, making $X^s$ an 
$N \times |C^s|$ matrix and $Y^s$ an $M \times |C^s|$ matrix.

We employ the concept of efficiency from DEA, and
the objects within a category are our decision making 
units (DMUs). An object's efficiency within its category is measured 
as a ratio of weighted sums, with the numerator being an aggregate of the 
output-characteristics and the  denominator being an aggregate of the 
input-characteristics.  The primary  goal of DEA is to decide nonnegative 
weighting parameters to maximize an object's efficiency while maintaining 
that each object's efficiency is no greater than one with the same weights.
There are several standard variants of this central theme, and we select 
the canonical input oriented model with variable returns to scale~\cite{cooper2007}. 
We specifically chose to calculate efficiency with the dual of the LP that 
maximizes the efficiency score for object $t$, which is
\begin{equation} \label{eq-effScore}
E^{t} = \mbox{min} \left\{ c^T \eta : 
	A^{t} \eta \leq 0, \, B \eta = e, \, \eta \geq 0 \right\}, 
\end{equation}
provided that
\[ 
A^{t} = \left[ 
	\renewcommand{\arraystretch}{1.3}
	\renewcommand{\arraycolsep}{4pt}
	\begin{array}{c|c|c}
		-Y^s & y^{t} & 0 \\ \hline
		X^s & 0 & -x^{t} 
	\end{array}
	\right], \;
B = \left[ 
	\renewcommand{\arraystretch}{1.3}
	\renewcommand{\arraycolsep}{4pt}
	\begin{array}{c|c|c}
		e^T & 0 & 0 \\ \hline
		0 & 1 & 0 
	\end{array}
	\right], \;\mbox{ and }\;
c = \left(
	\begin{array}{c}
		0 \\ 0 \\ \vdots \\ 0 \\ 1
	\end{array}
	\right),
\]  
with $o^t \in C^s$ and $x^t$ and $y^t$ being the respective 
$t$-columns of $X^s$ and $Y^s$. The all  ones column vector 
is $e$, with its length being decided by the context of its
use. Note that
\begin{equation} \label{eq-etaHat}
\begin{array}{l}
\displaystyle
\hat{\eta\,}^T = \left(0, 0, \ldots, 0, 1, 0, \ldots, 0, 1, 1 \right) \\[4pt]
\hspace*{31.3pt} \mbox{position $t$ } \raisebox{4pt}{$\uparrow$}
\end{array}
\end{equation}
is always feasible, and hence, $E^t \le 1$. The value of $E^t$ 
is the efficiency score of object $t$ relative to its category, and if 
this score is less than one, then object $t$ is inefficient in its 
category. Inefficiency means that another object's efficiency score within
the category bests the efficiency score of object $t$ no matter how the 
inputs and outputs are weighted.

Our assumption of uncertain information indicates that $E^t$ is likely
imprecise and that an inefficiency could be the result of inaccurate data.
The uDEA problem in~\cite{ehrgott2018} addresses 
this concern by acknowledging and modeling uncertainty and by recasting the 
efficiency score as a robust linear program. The solution of the robust 
linear program maximizes an object's efficiency score by selecting a best 
collection of data from among the uncertain possibilities. Thus an 
inefficient object based on the original data might become efficient with 
a prudent and reasonable selection of new data from a set of
uncertain options. We note that several modes of
robust DEA are now proposed and that applications are increasing,
see~\cite{peykani2020} and its bibliography as a review, as well
as the more recent publications~\cite{salahi2021,TOLOO2022102583}.

We use an ellipsoidal model of uncertainty for each row of $A^t$, and denoting
the $i$-th row of $A^t$ by $A_i^t$, we set
\[ 
\cU_i^t(\sigma_i) = \{ A_i^t + \sigma_i \, u ^T R_i : ||u||_{2} \le 1 \}.
\]
The scalar $\sigma_i$ is assumed to be nonnegative, and the matrix $R_i$ has 
$|C^s|+2$ columns along with a compatible number of rows to align with the 
dimension of the vector $u$.  The matrix $R_i$ defines the structure of 
uncertainty, whereas $\sigma_i$ scales this structure to shrink or enlarge the 
amount of uncertainty.  

The format of $A^t$ imposes a concomitant format on $R_i$. The last two 
columns of $A^t$ duplicate the input- and output-characteristics of object
$t$, and $R_i$ must maintain this format. Assume for illustrative 
purposes that $t = 2 \in C^s$ corresponds with the second column of data
in $X^s$ and $Y^s$. Then, if
\[
X^s = \left[ 
	\renewcommand{\arraycolsep}{2pt}
	\begin{array}{cc} 1.2 \; , & 0.8 \end{array} \right]
\; \mbox{ and } \;
Y^s = \left[ 
 	\renewcommand{\arraycolsep}{4pt}
	\renewcommand{\arraystretch}{1.2}
	\begin{array}{cc} 0.9 & 0.5 \\ 0.6 & 0.7 \end{array} \right],
\]
we have
\[
A^2 = \left[ 
 	\renewcommand{\arraycolsep}{4pt}
	\renewcommand{\arraystretch}{1.2}
	\begin{array}{rr|r|r}
	-0.9 & -0.5 & 0.5 & \multicolumn{1}{c}{0} \\ 
	-0.6 & -0.7 & 0.7 & \multicolumn{1}{c}{0} \\ \hline
	 1.2 &  0.8 & \multicolumn{1}{c|}{0}   & -0.8
	\end{array} \right].
\]
Suppose our model of uncertainty for the inputs assumes
a $2 \times 2$ identity with regard to $X^s$. Such diagonal models are 
commonly motivated by probabilistic assumptions such as the characteristics 
being uncorrelated random variables.  We have in this case that
\begin{equation} \label{eq-identityR}
R_3 = \left[ 
 	\renewcommand{\arraycolsep}{4pt}
	\renewcommand{\arraystretch}{1.2}
	\begin{array}{rr|r|r}
	1 & 0 & 0 & 0 \\
	0 & 1 & 0 & -1 
	\end{array} \right],
\end{equation}
from which 
\[
A^2_3 + \sigma_3 u^T R_3 = 
	\left[ \begin{array}{cccc}
	1.2 + \sigma_3 u_1, & 0.8 + \sigma_3 u_2, & 0, & -0.8 - \sigma_3 u_2
	\end{array} \right].
\]
The last column of $R_3$ must be the negative of the second column for
the structures of $A^2_3$ and $A^2_3 + \sigma_3 u^T R_3$ to agree as $u$ varies
over the unit disc.  Violations to this structural agreement would
essentially, and erroneously, equate to the second input-characteristic of the 
second object having two values as the data ranged over the uncertainty set. 
Similar requirements for the output-characteristics are necessary, with 
illustrative (non-identity) options for $A^2$ being
\[
R_1 = \left[ 
 	\renewcommand{\arraycolsep}{4pt}
	\renewcommand{\arraystretch}{1.2}
	\begin{array}{rr|r|r}
	-0.2 &  0.1 & -0.1 & 0 \\
 	 0.3 & -0.2 &  0.2 & 0
	\end{array} \right]
\; \mbox{ and } \;
R_2 = \left[ 
 	\renewcommand{\arraycolsep}{4pt}
	\renewcommand{\arraystretch}{1.2}
	\begin{array}{rr|r|r}
	0.7 \,, &  0.6 & -0.6 & 0 \\
	0.1 \,, & -0.3 &  0.3 & 0
	\end{array} \right].
\]
We express $R_i$ as $[R'_i \,|\, R''_i]$ in the forthcoming discussion,
with $R''_i$ having two columns and being decided by $R'_i$ as
illustrated.

We let $\sigma$ be the $M+N$ vector whose $i$-th component is $\sigma_i$, and 
we define the robust efficiency score for object $t$ with respect to 
$\sigma$ as,
\begin{eqnarray*}
\cE^t(\sigma) 
    & = & \mbox{min}  \left\{c^T \eta : \hat{A}^t_i \eta \leq 0, \; 
		\forall \, \hat{A}^t_i \in \cU^t_i(\sigma_i), \; 
		\forall \, i, \, B \eta = e, \, \eta \geq 0 \right\} \\[4pt]
    & = & \mbox{min}  \left\{c^T \eta : \hat{A}^t_i \eta 
            + \sigma_i \, \| R_i \eta \| \leq 0, \; \forall \, i, \, 
            B \eta = e, \, \eta \geq 0 \right\}, \\
\end{eqnarray*}
where the second math program is the standard second-order cone
expression of the robust problem. The structure of each $R_i$ 
mandates that $\cE^t(\sigma) \le 1$ for the same reason that 
$E^t \le 1$, i.e. because $\hat{\eta}$ in~\eqref{eq-etaHat} 
remains feasible. 

We require that $\cE^t(\sigma) = 1$ for sufficiently large
$\sigma$, a property that ensures object $t$ is capable of
efficiency, or more succinctly capable, if enough uncertainty is
assumed. A mathematically contrived example
in~\cite{ehrgott2018} without ellipsoidal uncertainty shows 
that $\sup \{ \cE^t(\sigma) : \sigma \ge 0 \}$ can be strictly less 
than one, and hence, capability is not necessarily ensured and requires 
applicable justification. Theorem~\ref{thm-capability} gives serviceable
assurance.
\begin{theorem} \label{thm-capability}
If $R'_i$ has full column rank for each $i$, then $\cE^t(\sigma) = 1$ for 
all sufficiently large $\sigma$.
\end{theorem}
\begin{proof}
Proposition 1 in~\cite{ehrgott2018} establishes the following 
general monotonicity property.  If two collections of uncertainty 
satisfy $\cU'_i \subseteq \cU''_i$ for each $i$, then
\begin{equation} \label{eq-monotonic}
\left.
\begin{array}{rcl}
\lefteqn{ \mbox{min}  \left\{c^T \eta : \hat{A}^t_i \eta \leq 0, \; 
    \forall \, \hat{A}^t_i \in \cU'_i(\sigma_i), \; 
	\forall \, i, \, B \eta = e, \, \eta \geq 0 \right\} } \\[6pt]
& \le &
\mbox{min}  \left\{c^T \eta : \hat{A}^t_i \eta \leq 0, \; 
    \forall \, \hat{A}^t_i \in \cU''_i(\sigma_i), \; 
	\forall \, i, \, B \eta = e, \, \eta \geq 0 \right\} \\[6pt]
& \le & 1,
\end{array}
\hspace*{8pt}
\right\}
\end{equation}
where the last inequality follows from an assumption that
$\hat{\eta}$ in~\eqref{eq-etaHat} is feasible in both problems.

Theorem 2 in~\cite{ehrgott2020} guarantees the existence of
$\hat{\sigma} > 0$ such that $\cE^t(\sigma) = 1$ if
$\sigma \ge \hat{\sigma}$ and if
\[
\cU_i = \{ A_i : A_i^t + \sigma u^T [I \,|\, R'_i]: \| u \|_{\infty} \le 1 \}.
\]
So the inequalities in~\eqref{eq-monotonic} will complete the proof upon
showing that for each $i$ and for sufficiently large $\sigma$,
\[
\{ A_i : A_i^t + \hat{\sigma}_i u^T [I \,|\, R''_i] : \| u \|_{\infty} \le 1 \}
\subseteq
\{ A_i : A_i^t + \sigma_i u^T R_i : \| u \|_{2} \le 1 \}.
\]
The required format of $R''_i$ in both cases reduces this argument 
to showing that 
\[
\{ \hat{\sigma}_i \, u^T : \| u \|_{\infty} \le 1 \}
\subseteq
\{ \sigma_i \, u^T R'_i : \| u \|_{2} \le 1 \},
\]
for each $i$ and sufficiently large $\sigma$.

Fix $i$ and select 
$v \in \{ \hat{\sigma}_i \, u^T : \| u \|_{\infty} \le 1 \}$.
Then $\| v \|_2 \le \hat{\sigma}_i^{\,|C^s|}$.
Further select $\sigma_i$ so that its components satisfy
\[
\sigma_i \ge \hat{\sigma}_i^{\,|C^s|} \, \left\| \left( R'_i \right)^+ \right\|_2,
\]
where $\left( R'_i \right)^+$ is the Moore-Penrose pseudo-inverse
so that 
\[
u^T = \frac{1}{\sigma_i} \; v^T \left( R'_i \right)^+ \; \mbox{ solves } \;
u^T \left( \sigma_i R'_i \right) = v^T.
\]
We now have
\[
\| u \|_2
\le \frac{1}{\sigma_i} \, \| v \|_2  \, 
        \left\| \left( R'_i \right)^+ \right\|_2 
\le \frac{1}{\sigma_i} \, \hat{\sigma}_i^{\,|C^s|} \, 
        \left\| \left( R'_i \right)^+ \right\|_2 
\le 1,
\]
which establishes the result. 
\end{proof}
An immediate consequence of Theorem~\ref{thm-capability} is that
invertible $R'_i$ matrices ensure capability, and since common
examples set $R'_i$ to be positive diagonal matrices or covariance matrices,
our assumption of capability is routine. Note also that the rank
condition is sufficient but not necessary, and that examples not
satisfying the rank condition can still result in each object being 
capable. Indeed, all but cleverly constructed examples lead to the 
capability of each object in the authors' experience, making the 
assumption of capability prevalent and not inimical to practical 
application.

We classify objects into categories so that the objects within
a category require a minimum amount of uncertainty for each of
the objects to have a claim to efficiency. So our motivating
question is, how much uncertainty does a category need for
all of its objects to have a robust efficiency score of one? 
We define the proximity to equitable efficiency for category 
$s$ to answer this question, which is
\begin{equation} \label{eqn: orig_model}
\cP^s = 
\mbox{min} \left\{ \| \sigma\| : \cE^t(\sigma) = 1, 
	\; \forall \, t \in C^s \right\}. 
\end{equation}
The constraints of this problem necessitate that each object 
reaches its maximum efficiency score by selecting characteristics 
from the uncertainty sets $\cU^t_i(\sigma_i)$, a requirement that 
imparts equity among the objects within the category relative 
to efficiency. Note that our proximity model does not require 
each object to be assessed in the same fashion because each 
object can combine their characteristic values differently to 
reach an efficiency score of one. A straightforward interpretation 
of $\cP^s$ is that it defines a minimum vicinity over which the
characteristic data can vary so that each object is equally and 
maximally efficient if it were allowed to select values from the 
uncertain possibilities. We comment that a different sense of equity 
appears in the multiple objective 
literature~\cite{baatar2006,kostreva1999}.

We say that a nonnegative $\sigma$ is feasible-for-classification 
for category $s$ if $\cE^t(\sigma) = 1$ for all $t \in C^s$, that is, 
if $\sigma$ is feasible for the optimization problem defining $\cP^s$.
Note that we have two properties from~\eqref{eq-monotonic}, those being:
\begin{center}
\begin{tabular}{rl}
Property 1: & if $\sigma'' \ge \sigma'$ and $\cE^t(\sigma') = 1$, then
	$\cE^t(\sigma'') = 1$, and \\[8pt]
Property 2: & if $\sigma'$ is feasible-for-classification and 
	$\sigma'' \ge \sigma'$, \\[2pt]
	& then $\sigma''$ is also feasible-for-classification.
\end{tabular}
\end{center}
Property 1 shows that object $t$ maintains an efficiency score of one if 
$\sigma$ surpasses what is required to satisfy $\cE^t(\sigma) = 1$, 
which suggests that we can construct a feasible-for-classification $\sigma$ 
by computing the componentwise maximums of solutions to
\[
\min \left\{ \| \sigma \| : \cE^t(\sigma) = 1 \right\}.
\]
Theorem~\ref{thm-boundProximity} confirms this suggestion and establishes 
both upper and lower bounds on $\cP^s$ based on this construction.  

Let
\[
{\bf \Sigma} = \left(\sigma^1, \sigma^2, \ldots, \sigma^{|C^s|} \right)
	\in \prod_{t\in C^s} \arg\!\min \left\{ \| \sigma \| 
		: \cE^t(\sigma) = 1 \right\},
\]
and let $\hat{\sigma}({\bf\Sigma})$ be the nonnegative vector
whose components are
\[
\hat{\sigma}_i ({\bf\Sigma}) = \max_t \left\{ \sigma_i^t \right\}.
\]
Theorem~\ref{thm-boundProximity} establishes upper and lower bounds
on $\cP^s$ in terms of $\hat{\sigma}({\bf\Sigma})$.

\begin{theorem} \label{thm-boundProximity}
We have for any
\[
{\bf \Sigma} \in \prod_{t \in C^s} \arg\!\min \left\{ \| \sigma \|
                : \cE^t(\sigma) = 1 \right\}
\]
that the proximity to equitable efficiency for category $s$ satisfies
\[
\frac{1}{\sqrt{M+N}} \, \| \hat{\sigma} ({\bf\Sigma}) \| \le \cP^s
	\le \| \hat{\sigma} ({\bf\Sigma}) \|.
\]
\end{theorem}
\begin{proof}
Select
\[
{\bf \Sigma} = \left(\sigma^1, \sigma^2, \ldots, \sigma^{|C^s|} \right) 
        \in \prod_{t \in C^s} \arg\!\min \left\{ \| \sigma \|
                : \cE^t(\sigma) = 1 \right\}.
\]
Then $\hat{\sigma} ({\bf\Sigma}) \ge \sigma^t$ for all $t$ by definition,
and we have from Property 1 above that $\cE^t(\hat{\sigma}) = 1$ 
for all $t$. So $\hat{\sigma} ({\bf\Sigma})$ is feasible-for-classification
for category $s$
and 
\[
\cP^s \le \| \hat{\sigma} (\bf\Sigma) \|.
\]

The lower bound is established by first noticing that 
\[
\sigma_i^t \le \| \sigma^t \| = 
\min \left\{ \| \sigma \| : \cE^t(\sigma) = 1 \right\} \le \cP^s
\; \mbox{ for all $i$ and $t$.}
\]
So,
\[
\hat{\sigma}_i ({\bf\Sigma}) = \max_t \; \sigma_i^t \le \cP^s
	\,\; \mbox{ for all $i$}.
\]
We conclude that $\hat{\sigma} ({\bf\Sigma}) \le \cP^s \, e$,
and hence,
\[
\| \hat{\sigma} ({\bf\Sigma}) \| \le \sqrt{M+N} \; \cP^s,
\]
which completes the proof. 
\end{proof}
Two corollaries follow, the first of which is an immediate consequence of
Theorem~\ref{thm-boundProximity}.
\begin{corollary} \label{cor-supinfbound}
The proximity to equitable efficiency for category $s$ satisfies,
\[
\frac{1}{\sqrt{M+N}} \, \left( 
	\sup_{\bf\Sigma} \, \left\{ \| \hat{\sigma} ({\bf\Sigma}) \| 
		\right\} \right)
	\le \cP^s \le \inf_{\bf\Sigma} \, 
		\left\{ \| \hat{\sigma} ({\bf\Sigma}) \| \right\},
\]
where the supremum and infemum are expressed over the collection
\[
{\bf\Sigma} \in \prod_{t \in C^s} \arg\!\min \left\{ \| \sigma \|
                : \cE^t(\sigma) = 1 \right\}.
\]
\end{corollary}
Theorem~\ref{thm-boundProximity} can also be strengthened to show that 
$\cP^s$ achieves its upper bound should $\hat{\sigma} ({\bf\Sigma})$ 
be decided by a single object.
\begin{corollary} \label{cor-uDeaReduction}
If $\hat{\sigma} ({\bf\Sigma}) \in \arg\!\min \left\{ \| \sigma \| : 
	\cE^t(\sigma) = 1 \right\}$ for some $t \in C^s$,
then 
\[
\cP^s = \| \hat{\sigma} ({\bf\Sigma}) \|.
\]
\end{corollary}
\begin{proof}
Let $\hat{t} \in C^s$ have the property that
$\hat{\sigma} ({\bf\Sigma}) \in \arg\!\min \left\{ \| \sigma \| : 
\cE^{\hat{t}}(\sigma) = 1 \right\}$.  Then
\[
\| \hat{\sigma} ({\bf\Sigma}) \| 
	= \min \left\{ \| \sigma \| : \cE^{\hat{t}}(\sigma) = 1 \right\} \\
	\le \min \left\{ \| \sigma \| : \cE^{t}(\sigma) = 1, 
		\; \forall \; t \in C^s \right\} \\
	= \cP^s,
\]
and the result follows because Theorem~\ref{thm-boundProximity} gives
the reverse inequality.
\end{proof}

We address how these mathematical results guide our computational 
effort to compute $\cP^s$ in the next section. However, we first note 
that we assess a classification by summing the proximity values of the 
individual categories, giving an overall proximity evaluation
for the entire classification. So we appraise a classification 
$\cC$ with
\[
\cP(\cC) = \sum_{s=1}^S \; \cP^s,
\]
and its value indicates how much uncertainty is needed so that
all objects have a claim to efficiency within their categories.
Classifications with low total proximity values are preferred
because they partition objects into sets that are intra-category 
efficient relative to low amounts of uncertainty.

\section{Computing the Proximity to Equitable \\ Efficiency} \label{sec-algorithm}

Solving uDEA problems is generally difficult due to a loss of convexity,
see Example 3 in~\cite{ehrgott2018} as an illustration, and calculating 
the proximity to equitable efficiency inherits this difficulty because 
Corollary~\ref{cor-uDeaReduction} equates the value of $\cP^s$ to
solving a uDEA problem under appropriate conditions. A first order algorithm 
in~\cite{ehrgott2018} successfully solves uDEA problems, and we use 
this algorithm to compute a $\bf\Sigma$ by iteratively solving for 
each $t \in C^s$,
\[
\min \left\{ \| \sigma \| : \cE^t(\sigma) = 1 \right\}.
\]
The resulting $\hat{\sigma} ({\bf\Sigma})$ then bounds $\cP^s$
according to Theorem~\ref{thm-boundProximity}, and it even computes 
$\cP^s$ if $\hat{\sigma} ({\bf\Sigma})$ is defined by a single 
object, in which case $\cP^s = \| \hat{\sigma} ({\bf\Sigma}) \|$.

We work to calculate $\cP^s$ by reducing $\hat{\sigma} ({\bf\Sigma})$
if $\cP^s$ is not decided by a single object. Let $\Gamma(\sigma)$ be the 
total sum of robust efficiency scores within the category, 
\[
\Gamma(\sigma) = \sum_{t \in C^s} \cE^t(\sigma).
\]
Then,
\[
\cP^s = \min \left\{ \| \sigma \| : \cE^t(\sigma) = 1, \; 
    \forall \, t \in C^s \right\}
= \min \left\{ \| \sigma \| : \Gamma(\sigma) = |C^s| \right\},
\]
where the last equality follows from the fact that $\Gamma(\sigma)$ 
is equal to the number of objects in $C^s$ if and only if $\sigma$ is 
feasible-for-classification. A unit-length improving direction $d$ from 
$\sigma$ is calculated by solving,
\begin{equation} \label{eq-dirSrch}
\min \left\{ d^T \sigma : 
    d^T \nabla \Gamma(\sigma) \ge 0, \; d^T d \le 1 \right\}.
\end{equation}
The first constraint ensures that $\sigma + \alpha d$ is marginally
feasible-for-classification, and the objective minimizes the directional
derivative of $\| \sigma \|^2$ along $d$ -- note that we have removed the
superfluous scalar of two. 

We search along the solution $d$ if the optimal value of~\eqref{eq-dirSrch} 
is negative, and a straightforward calculation in this case shows that
the step size $\alpha$ must satisfy
\begin{equation} \label{eq-alphaBnd}
0 \le \alpha \le -\sigma^T d -
		\sqrt{\left( \sigma^T d \right)^2 - 
			\left( \| \sigma \|^2 -
			\frac{1}{M+N} \, \| \hat{\sigma} ({\bf\Sigma}) \|^2
			\right) }
\end{equation}
to ensure that $\sigma + \alpha \, d$ does not violate the lower bound in 
Theorem~\ref{thm-boundProximity}. We search this interval with the
method of bisection to find the largest value of $\alpha$ guaranteeing
$\Gamma(\sigma + \alpha \, d) \ge |C^s| - \varepsilon$, where $\varepsilon$ is
a suitably small tolerance required by our computational study.  We then 
update $\sigma$, solve~\eqref{eq-dirSrch}, and repeat our search for 
$\alpha$. The process starts with $\sigma = \hat{\sigma} ({\bf \Sigma})$ 
and terminates once the solution to~\eqref{eq-dirSrch} is nonnegative.

We assumed the existence of $\nabla\Gamma$ in~\eqref{eq-dirSrch}, but this
vector is approximated componentwise with the finite differences
\begin{equation} \label{eq-grad}
\frac{\partial\Gamma}{\partial \sigma_i} 
	\approx \frac{ \Gamma(\sigma) - \Gamma(\sigma - \delta e_i) }{\delta},
\end{equation}
where $e_i$ is the vector of zeros except for a one in the $i$-th position.
We note that this approximation exists even if the gradient doesn't.
The use of a backward difference is important since the more common
forward difference would instead use $\Gamma(\sigma + \delta e_i)$ 
and $\Gamma(\sigma)$. However, $\sigma$ being feasible-for-classification
implies that $\sigma + \delta e_i$ would also be feasible-for-classification, 
and hence, both $\Gamma(\sigma + \delta e_i)$ and $\Gamma(\sigma)$ would 
have the same value and our approximate gradient would be zero. The 
search direction from~\eqref{eq-dirSrch} would then be 
$d = -\sigma / \| \sigma \|$, making 
\[
\sigma + \alpha d = \left( 1 - \frac{\alpha}{\| \sigma \|} \right) \sigma.
\]
The misguided computational interpretation would be that 
we could approximate $\cP^s$ by calculating any 
$\hat{\sigma} ({\bf\Sigma})$ and then scaling this vector
down while satisfying 
\[
\Gamma\left( \left(1 - \frac{\alpha}{\| \hat{\sigma} ({\bf\Sigma}) \|} \right)\, 
    \hat{\sigma}(\bf\Sigma)  \right) \ge |C^s| - \varepsilon.
\]
So the use of a forward difference would limit the search to 
compute $\cP^s$ from an initial $\hat{\sigma} ({\bf\Sigma})$ to a single 
search direction due to the erroneous assumption that $\Gamma(\sigma)$ 
is constant if $\sigma$ is in some small neighborhood of 
$\hat{\sigma} ({\bf\Sigma})$. The use of the backward differences 
removes this concern and provides a more accurate search strategy. 

Our algorithmic approach to compute $\cP^s$ is: \\[-10pt]
\begin{center}
\renewcommand{\arraystretch}{1.6}
\begin{tabularx}{0.92\textwidth}{clX}
Step 0: & Initialize &  Use the algorithm in~\cite{ehrgott2018} to 
	compute a $\bf\Sigma$ and a resulting $\hat{\sigma} ({\bf\Sigma})$.
	If $\hat{\sigma}({\bf\Sigma})$ is defined by a single
	object, then set $\cP^s = \|\hat{\sigma}({\bf\Sigma}) \|$
	and stop. Otherwise, initialize 
	$\sigma$ to $\hat{\sigma} ({\bf\Sigma})$. \\[8pt]
Step 1: & Search Direction & Compute an approximate $\nabla \Gamma(\sigma)$
	with~\eqref{eq-grad} and then use this approximation to
	calculate a search direction $d$ by solving~\eqref{eq-dirSrch}.
	Stop if $d^T \sigma \ge 0$. \\
Step 2: & Line Search &  Use the method of bisection to calculate the 
	largest $\alpha$ satisfying 
	$\Gamma(\sigma + \alpha d) \ge S - \varepsilon$,
	where $\alpha$ satisfies~\eqref{eq-alphaBnd}.
	Update $\sigma$ to $\sigma + \alpha d$ and return to 
	Step 1. 
\end{tabularx} \\[8pt]
\end{center}
We could add a termination criterion to stop if 
$\Gamma(\sigma + \alpha d)$ reached its lower bound in Step 2; however,
theory indicates that the directional derivative from~\eqref{eq-dirSrch}
would be nonnegative in this case, and we use the termination criterion
in Step 1 to validate our computational scheme's agreement with theory.
This approach has been computationally successful.

We comment that each evaluation of $\Gamma(\sigma)$ requires $|C^s|$ robust 
optimization problems to be solved, and while the algorithmic structure 
above is standard, assembling the computational effort is less so.  Robust 
solvers are increasingly trustworthy, but still, they lack some of the 
dependability associated with other problem classes like linear 
programming~\cite{bental2001,holder2020}.
We use Gurobi with many of the numeric options set to their 
most stringent possibilities, and we validate the optimal status of each 
solve.  Success is routine with the increased numeric focus but is more
questionable with default settings. If a problem fails our stringent 
expectations, then we iteratively relax convergent tolerances while 
guaranteeing optimal values, which are efficiency scores, to at least
six decimal places -- so an efficiency score of at least 0.999999 equates
to 1.0. The vast majority of solves more stringently reaches 16 decimal 
places of accuracy.

\section{Classification with Uncertain Data} \label{sec-classification}

Our classification problem partitions the $T$ objects in
$\mathcal{O}$ into the $S$ categories in $\cC$ in a way that 
(approximately) minimizes the total proximity $\cP(\cC)$, 
i.e. we seek to solve
\begin{equation} \label{eq-classificationProb}
\min \{ \cP(\cC) : |\cC| = S \}.
\end{equation}
The search space of this problem is combinatorially problematic 
for even modest problems because its size is the Sterling number of
the second kind that counts the number of ways to disjointly 
partition $T$ objects into $S$ nonempty sets. To illustrate, the search
spaces of the examples in the next section are
\[
\left\{ \begin{array}{c} 30 \\ 3 \end{array} \right\}
    \approx 3.4 \times 10^{13}
\mbox{ and } 
\left\{ \begin{array}{c} 42 \\ 3 \end{array} \right\}
    \approx 1.8 \times 10^{19}.
\]
The immensity of these spaces fundamentally
challenges a classification scheme independent of how
we assess individual classifications, which prompts
the need for heuristics to navigate a search.
We divide our search for a (near) optimal
classification into an initialization procedure, which
efficiently solves several adapted $p$-median problems,
and an iterative improvement strategy, which is a neighborhood 
search that reassigns one object per iteration.

The $p$-median problem partitions objects into sets so that 
the aggregate distance from the objects within the sets to 
their representative medians is as small as possible. We
begin by considering the entire collection of objects
as a single classification, i.e. we do not assume a
partition and originally set $\cC = \{ \mathcal{O}\}$. 
The $p$-median problem depends on a distance between 
the objects, and we let the distance between objects 
$t_i$ and $t_j$ be
\[
d(t_i, \, t_j) = | \, \|\sigma^i\| - \|\sigma^j\| \, |,
\]
where $\sigma^i$ and $\sigma^j$ respectively solve
\[
\min \{ \| \sigma \| : \cE^{t^i}(\sigma) = 1 \}
\; \mbox{ and } \;
\min \{ \| \sigma \| : \cE^{t^j}(\sigma) = 1 \}.
\]
Again, these $\sigma$ vectors are relative to the
entire collection of objects and not to a proposed partition 
of $S$ categories; so there is no intra-category consideration
in the initialization process. The idea is that calculating the 
minimum amount of uncertainty for each object against the whole
set of objects should be a reasonable proxy to seed our search 
to minimize the sum of the intra-category uncertainties 
comprising $\cP(\cC)$.

We adapt the $p$-median problem to include constraints that
mandate the sizes of the sets.  For instance, if we seek
to classify $30$ objects into three categories, then we solve 
an adapted $p$-median problem for sets whose sizes are 
$(2,2,26)$, $(2,3,25)$, $(2,4,24)$, etc.  If the set sizes 
are $(p_1, p_2, \ldots, p_S)$, then we solve
\[
\renewcommand{\arraystretch}{1.3}
\begin{array}{rrcl}
\min & \multicolumn{3}{l}{\sum\limits_{ij} d(t_i, t_j) \, \zeta_{ij}} \\
\mbox{subject to}
& \sum\limits_j \zeta_{ij} & = & 1, \; \forall i \\
& \sum\limits_j \zeta_{jj} & = & S, \\
& \zeta_{jj} & \ge & \zeta_{ij}, \; \forall i,j, \, i \neq j \\
& \sum\limits_k \omega_{jk} & \le & 1, \; \forall j \\
& \sum\limits_i \zeta_{ij} & = & \sum\limits_{k} p'_k \, \omega_{jk}, 
    \; \forall j \\
& \sum\limits_j \omega_{jk} & = & | \, \{ p_q : p_q = p'_k\} \, |,
    \; \forall k \\
& \zeta_{ij} & \in & \{0, 1\}, \; \forall i,j \\
& \omega_{jk} & \in & \{0, 1\}, \; \forall j,k,
\end{array}
\]
where $i$ and $j$ index over $\{1, 2, \ldots, S\}$,
$(p'_1, p'_2, \ldots, p'_{S'})$ lists the unique elements of
$(p_1, p_2, \ldots, p_S)$, and $k$ indexes over
$\{1, 2, \ldots, S'\}$. The first three constraints model
the traditional $p$-median problem, whereas the next three
ensure that the partitioning sets have the specified 
cardinalities. We note that we do not need to consider singletons
because any set with only one or two elements trivially has 
a proximity value of zero.  So we solve the adapted $p$-median 
for all other summative partitions of $S$, and we calculate 
$\cP(\cC)$ for each. Our initial classification is the one with 
the smallest value of $\cP(\cC)$.

Suppose $\cC = \{C^1, C^2, \ldots, C^S\}$ is our initial
incumbent classification. We define a neighbor of $\cC$ to 
be any classification that moves a single object from one category
to another.  So $\cC'$ is a neighbor of $\cC$ if for some 
$s'$ and $s''$, and for some object $o^t \in C^{s'}$,
\[
\cC' = \left( \cC \backslash \{ C^{s'}, C^{s''} \} \right)
    \bigcup \left\{ C^{s'} \backslash \{ o^t \}, \,
        C^{s''} \cup \{ o^t \} \right\}.
\]
From the current incumbent classification $\cC$, we solve
\begin{equation} \label{eq-swapandseek}
\min \{ \cP(\cC') : \cC' \mbox{ is a neighbor of } \cC \}
\end{equation}
by iteratively swapping each object from its current category
to the other categories. We accept a solution 
to~\eqref{eq-swapandseek} as the next incumbent if its value is 
less than $\cP(\cC)$, but we otherwise terminate with a final 
classification. This process repeats with the updated
incumbent if improvement is found.

We conclude with a few comments about our computational process.
Our initialization and subsequent neighborhood search does not 
guarantee a solution to~\eqref{eq-classificationProb}. However, our
adapted $p$-median problem efficiently and efficaciously 
seeds our neighborhood search algorithm so that it converges in a few 
iterations. Moreover, and beyond combinatorial concerns, shrewd 
and judicious navigation through the collection of all possible 
classifications is prudent due to the computational burden of 
calculating $\cP(\cC)$, which requires thousands of SOCP solves 
even on modestly sized problems. Our two phase approach demonstrates 
persuasively on the examples of the next section, and we advocate
it as a reasonable and practically justified approximation
algorithm. Lastly, we have experimented with numerous minor 
variations, e.g. by adjusting the distances in our adapted
$p$-median problem to $d(t_1, t_2) = \| \sigma^1 - \sigma^2 \|_p$ 
or by tweaking the definition of a classification's neighborhood. We 
considered such adjustments in an attempt to expediting our 
computational task without degrading fidelity to the goal of 
identifying an optimal classification, but all adjustments 
either diminished outcomes or lacked computational improvement.

\section{Illustrative Examples} \label{sec-examples}

We apply our classification scheme to two problems to demonstrate
its employ. The first example partitions the 30 stocks in the 
Down Jones Industrial Average into three performance tiers, and
the second categorizes 42 prostrate treatments that were originally 
analyzed with DEA in~\cite{liu2013} and subsequently with 
uDEA in~\cite{ehrgott2018,ehrgott2020}. The objects, i.e. stocks
and prostate treatments, have one input and one output characteristic
in each case. Single input and output examples lend themselves 
to graphical inspection to help verify efficacy, and they guarantee 
that the upper and lower bounds on $\cP(\cC)$ are within $\sqrt{2}$
of each other from Theorem~\ref{thm-boundProximity}. Both examples
set $R'_i = I$, so a feasible-for-classification $\sigma$ exists
from Theorem~\ref{thm-capability}.

A few comments about the computational burden are worthwhile.
Our code is written in python and is available as a supplement 
to this article. We use Pyomo~\cite{hart2017} to model 
all optimization problems and Gurobi to solve all models.
Our initial computational testing suggested that both
examples would take about a week of serial calculation 
on a modern laptop, which seemed extreme. A parallel version
of the code that distributes the calculation over twelve nodes
has significantly reduced the calculation time 
to several hours, with the stock example solving in about
six hours and the prostate example solving in about eight.
The computational burden of classifying with regard to uncertain 
data is nonetheless significant and requires careful consideration.

The input and output characteristics of each stock are
respectively the semi-deviation and the average annual return,
both of which are uncertain estimates of future values.
Semi-deviation is a common measure of risk and is the standard 
deviation of downside returns, whereas the average annual
return is a routine measure of reward.  The risk reward
trade-off is steeped in investing, with risk adverse 
individuals accepting lower average returns to help safeguard 
principal and risk accepting individuals jeopardizing principal 
as they pursue higher returns. The input and output 
characteristics are calculated from the twelve annual 
returns ending at the beginning of each month in 2021.

Classifying stocks relative to efficiency with regard to
risk and reward does not preference risk or reward but instead
considers them as uncertain co-equal trade-offs, and hence, 
the categories tier stocks into performant classes
independent of risk tolerance. A reasonable interpretation 
is that Tier 1 stocks should be considered over Tier 2
or 3 stocks regardless of one's palatability 
toward risk. Tier 2 stocks should likewise be considered over 
Tier 3 stocks.

Figures~\ref{fig-stock1} through~\ref{fig-stock4} show the
progression of our classification process. Tier 1 stocks are
green squares, Tier 2 stocks are blue squares, and Tier 3 stocks
are red diamonds. The outcome from the initialization process 
is in Figure~\ref{fig-stock1}, and we note that Tier 3
can obviously add stocks from Tier 2 while maintaining
a proximity value of zero. The first two  improvements 
identify this fact and reduce the uncertainty of Tier 2 
by moving two of its stocks to Tier 3.  The third  and 
final improvement moves a Tier 1 stock to Tier 2, and 
the algorithm identifies no subsequent improvement.
Table~\ref{table-stockTiers} lists the categories throughout
the classification process along with the proximity values.
Stocks are listed by their ticker symbol.

\begin{figure}[t]
\begin{minipage}{0.45\linewidth}
    \begin{center}
    \includegraphics[width=\linewidth]{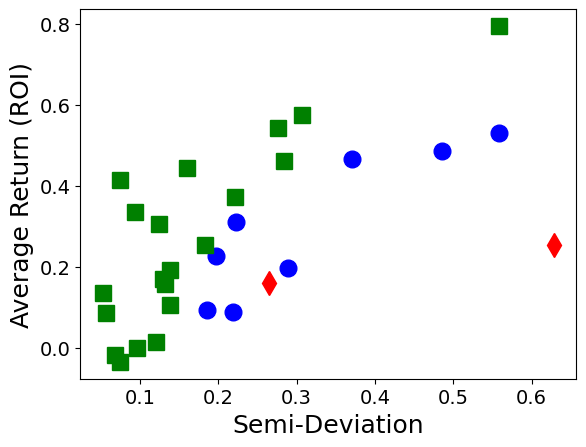}
    \end{center}
    \caption{Initial Classifications}
        \label{fig-stock1}
\end{minipage}
\hspace*{0.07\linewidth}
\begin{minipage}{0.45\linewidth}
    \begin{center}
    \includegraphics[width=\linewidth]{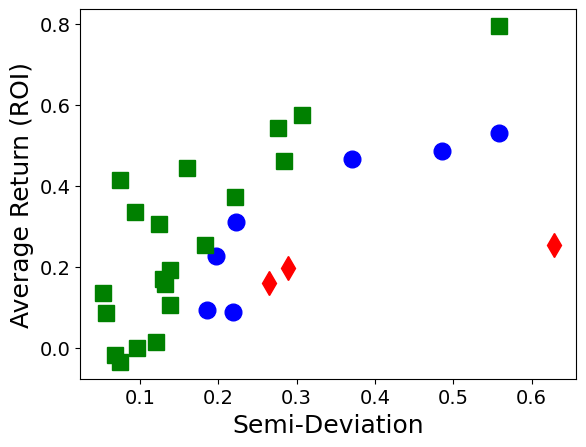}
    \end{center}
    \caption{First Improvement}
        \label{fig-stock2}
\end{minipage} \\[30pt]
\begin{minipage}{0.45\linewidth}
    \begin{center}
    \includegraphics[width=\linewidth]{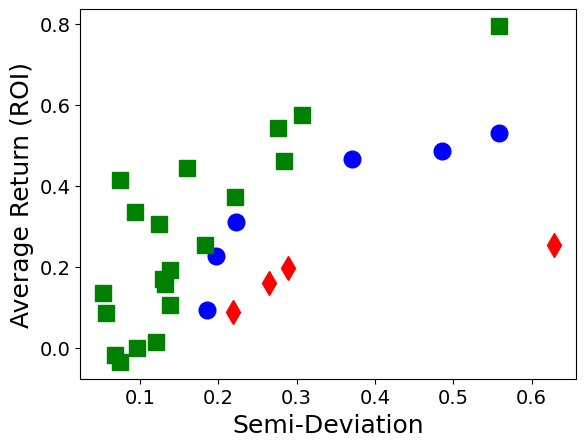}
    \end{center}
    \caption{Second Improvement}
        \label{fig-stock3}
\end{minipage}
\hspace*{0.07\linewidth}
\begin{minipage}{0.45\linewidth}
    \begin{center}
    \includegraphics[width=\linewidth]{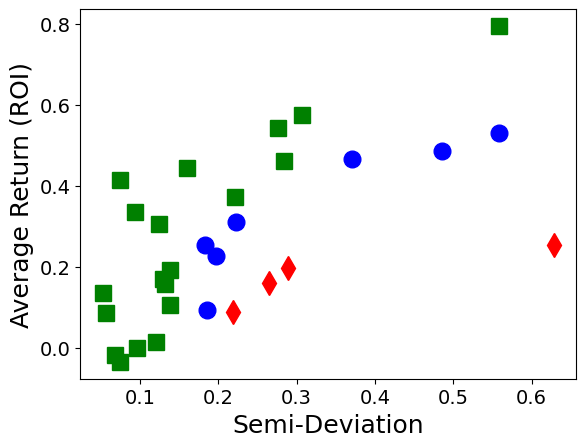}
    \end{center}
    \caption{Final Classification}
        \label{fig-stock4}
\end{minipage}
\end{figure}

\begin{table}[t!]
\begin{center}
{\footnotesize
\renewcommand{\tabcolsep}{4pt}
\renewcommand{\arraystretch}{1.2}
\begin{tabular}{l|l|l|l|c}
              & \multicolumn{1}{c|}{Tier 1} 
              & \multicolumn{1}{c|}{Tier 2}          
              & \multicolumn{1}{c|}{Tier 3}  
              & \multicolumn{1}{c}{$\cP(\cC)$} \\
\hline\hline & & & & \\[-10pt]
Initial       & \multicolumn{1}{c|}{$\cP^1=0.1361$}      
              & \multicolumn{1}{c|}{$\cP^2=0.0889$}   
              & \multicolumn{1}{c|}{$\cP^3 = 0.000$}  
              & $0.2250$ \\[4pt]
              & AXP,AMGN,AAPL,CAT,  & AXP,CVX,CSCO,KO, & BA,WBA  \\
              & DIS,GS,HD,HON,INTC, & DOW,IBM,JPM,TRV  &         \\
              & JNJ,MCD,MRK,MSFT,   &                  &         \\
              & NKE,PG,CRM,UNH,VZ,  &                  &         \\
              & V,WMT               &                  &         \\
\hline & & & & \\[-10pt]
Improvement 1 & \multicolumn{1}{c|}{$\cP^1=0.1361$}      
              & \multicolumn{1}{c|}{$\cP^2=0.0646$}   
              & \multicolumn{1}{c|}{$\cP^3 = 0.000$}  
              & $0.2001$ \\[4pt]
              & AXP,AMGN,AAPL,CAT,  & AXP,CSCO,KO,DOW  & BA,CVX, \\
              & DIS,GS,HD,HON,INTC, & IBM,JPM,TRV      & WBA     \\
              & JNJ,MCD,MRK,MSFT,   &                  &         \\
              & NKE,PG,CRM,UNH,VZ,  &                  &         \\
              & V,WMT               &                  &         \\
\hline & & & & \\[-10pt]
Improvement 2 & \multicolumn{1}{c|}{$\cP^1=0.1361$}      
              & \multicolumn{1}{c|}{$\cP^2=0.0172$}   
              & \multicolumn{1}{c|}{$\cP^3 = 0.000$}  
              & $0.1532$ \\[4pt]
              & AXP,AMGN,AAPL,CAT,  & AXP,CSCO,KO,DOW  & BA,CVX, \\
              & DIS,GS,HD,HON,INTC, & JPM,TRV          & IBM,WBA \\
              & JNJ,MCD,MRK,MSFT,   &                  &         \\
              & NKE,PG,CRM,UNH,VZ,  &                  &         \\
              & V,WMT               &                  &         \\
\hline & & & & \\[-10pt]
Final         & \multicolumn{1}{c|}{$\cP^1=0.1153$}      
              & \multicolumn{1}{c|}{$\cP^2=0.0209$}   
              & \multicolumn{1}{c|}{$\cP^3 = 0.000$}  
              & $0.1361$ \\[4pt]
              & AXP,AMGN,AAPL,CAT,  & AXP,CSCO,KO,DOW  & BA,CVX, \\
              & DIS,GS,HD,HON,INTC, & JPM,TRV,CRM      & IBM,WBA \\
              & JNJ,MCD,MRK,MSFT,V  &                  &         \\
              & NKE,PG,UNH,VZ,WMT   &                  &         \\
\end{tabular} }
\end{center}
\caption{Stock tiers through the classification process} \label{table-stockTiers}
\end{table}

The second example is a collection of 42 prostate cases that
were previously evaluated with DEA and uDEA, 
see~\cite{ehrgott2018,lin2013quality,ehrgott2020}. These
treatments were delivered to patients in New Zealand and
were approved for observational studies. Treatment planning
has an established history in operations research, and we refer 
readers to~\cite{EhGueHaShao08,EhrgottHolder2017,holder4,romeijn08} 
as summary reviews. The overriding
goal of treatment planning is to design a treatment that
delivers a uniform, tumoricidal dose to the targeted region
while sparing nearby organs from harmful effects. The targeted
region is the planning treatment volume (PTV), which is a 
clinically delineated portion of the anatomy that encompasses
both cancerous tumors and their perceived microscopic extensions. 
The organs at risk in a prostate case are generally the rectum,
the bladder, and the femoral heads, with the rectum and bladder 
being particularly concerning because they abut the prostate.

Each treatment is patient specific, and although trained experts 
tailor treatments with sophisticated optimization software to 
clinical standards, treatment quality remains nebulous and varied 
for many reasons. There are preference and software differences 
among clinics and clinicians, there are mathematical and computational 
assumptions that only approximate what is actually delivered, and there 
are uncertainties in treatment delivery. All 42 cases were deemed 
clinically acceptable, but that does not mean that they were bereft 
of improvement. Indeed, one of the outcomes of~\cite{lin2013quality}
was that treatments identified for possible improvement actually
benefited from reconsideration. So the practical importance of
reviewing previously delivered treatments is to learn from those 
that were outstanding so that clinical acceptability improves.

The original work in~\cite{lin2013quality} posited the use of
DEA to analyze treatments, but this effort assumed certain characteristic
data. The subsequent efforts in~\cite{ehrgott2018}
and~\cite{ehrgott2020} recognized the uncertainty of the
characteristics and provided alternate assessments that 
benefited from this recognition. The computational processes
in~\cite{ehrgott2018,ehrgott2020} evaluated individual treatments 
on the amount of uncertainty required to be efficient within the 
entire collection of treatments, and those that necessitated 
larger amounts of uncertainty were recommended for review and
possible improvement.

Our classification scheme advances the work
in~\cite{ehrgott2018,lin2013quality,ehrgott2020}.  First, it
recognizes and models uncertainty like~\cite{ehrgott2018,ehrgott2020},
so it gains fidelity over~\cite{lin2013quality} with regard 
to the uncertain clinical situation. Second, the analysis methods
in~\cite{ehrgott2018,ehrgott2020} are not directly classification
processes, but rather, they ask clinicians to discern which 
treatments might benefit from additional consideration post 
calculation. For instance, the authors of~\cite{ehrgott2018}
suggest segmenting the cases with a threshold on the amount of 
uncertainty, but this threshold can only be decided after the 
computational effort. Our classification process instead identifies 
(near) optimal categories and partitions the cases into a fixed 
number of categories without an after-the-fact cutoff. We partition 
the 42 cases into three categories: those deemed to be clinical exemplars, 
those deemed to be clinical standards, and those that should be 
considered for additional improvement. Classifying into three
categories is clinically reasonable because it identifies clinical
exemplars that illustrate best practices, which is unlike 
the previous dichotomies that primarily distinguish the treatments
that should be reconsidered. Our trichotomy does both, i.e. it 
identifies clinical exemplars and it distinguishes those that should 
be reconsidered.

The input characteristic is the rectal generalized equivalent 
uniform dose (gEUD), which measures the average homogeneous 
dose delivered to the rectum, and the output characteristic is the
dose received by 95\% of the PTV ($D_{95}$). The unit of dose is 
a Gray (Gy). These are the same input and output characteristics 
in~\cite{ehrgott2018,lin2013quality,ehrgott2020}, and all three of
these articles review additional clinical details. The
categories throughout the classification process are in
Figures~\ref{fig-prostate1} and~\ref{fig-prostate2}. The category
of green squares contains exemplar treatments, the category of
blue circles contains clinical standards, and the category of red
diamonds contains treatments that should be reconsidered.
Note that the search following the initialization process only
identifies a single improvement, with one of the treatments moving
into the category of clinical standards from those that would
have been recommended for continued review.

\begin{figure}[t]
\begin{minipage}{0.45\linewidth}
    \begin{center}
    \includegraphics[width=\linewidth]{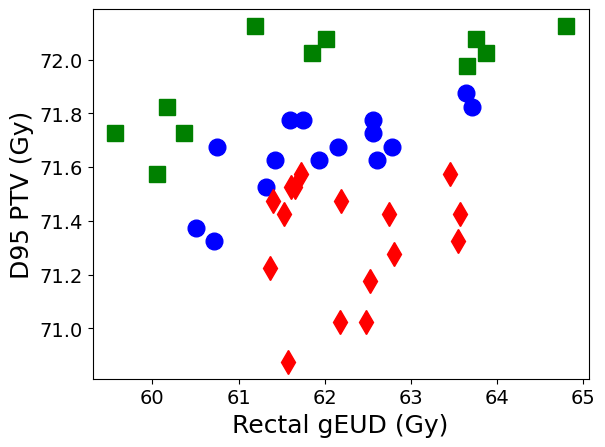}
    \end{center}
    \caption{Initial Classifications}
        \label{fig-prostate1}
\end{minipage}
\hspace*{0.07\linewidth}
\begin{minipage}{0.45\linewidth}
    \begin{center}
    \includegraphics[width=\linewidth]{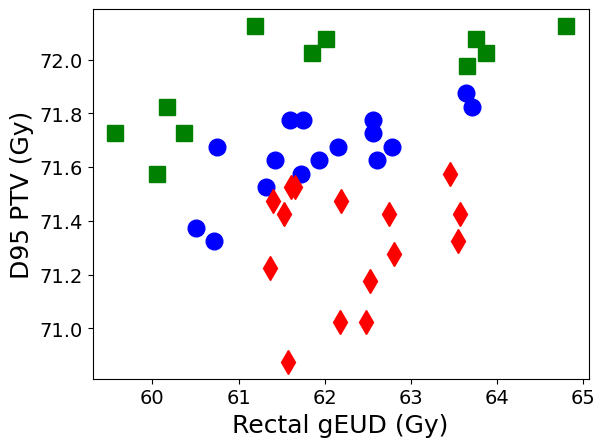}
    \end{center}
    \caption{First Improvement}
        \label{fig-prostate2}
\end{minipage}
\end{figure}

The initial and improved proximity values are: 
\begin{center}
\renewcommand{\tabcolsep}{6pt}
\renewcommand{\arraystretch}{1.2}
\begin{tabular}{l|cccc}
        & $\cP^1$  & $\cP^2$  & $\cP^3$  & $\cC(\cP)$  \\
\hline
Initial & $0.1851$ & $0.1693$ & $0.8049$ & $1.1593$ \\
Final   & $0.1851$ & $0.1693$ & $0.7997$ & $1.1541$.
\end{tabular}
\end{center}
Our partitioning suggests 15 cases would have likely benefited 
from additional planning, which is  fewer than  the 20 
cases shown to exceed the proposed threshold in~\cite{ehrgott2018}.
Some of this decrease is due, at least in part, to the fact 
that we partition into three categories instead of two. 
Three final observations are worthwhile. First, the interplay 
between the second and third categories suggests that 
it would be difficult to discern a clear separation between these
categories without a tool that optimizes the classification. Second,
just because a treatment is classified as warranting a review does not
immediately suggest that the treatment fell below clinical guidelines.
Remember that trained experts tailor each treatment to a patient's specific
needs, and sometimes this imposes considerations other than standard performance
metrics.  So being classified as warranting a review would simply alert an
expert to the fact that continued design might be worthwhile. Third,
the categories visually segment themselves more vertically than horizontally, 
which suggests that the $D_{95}$ value is somewhat more important as a
classification parameter.  This observation is less clear between the
second and third categories, but it is the type of visual acuity provided by
the classification.

\section{Conclusion} \label{sec-conclusion}

Uncertain data envelopment analysis can be useful within a classification
process, and we have advanced the employ of the proximity to equitable
efficiency as a sound tool to classify in the presence of uncertain data.
We overcome two computational issues, those being the innate difficulty
of calculating our proximity value due to a loss of convexity and the
combinatorially fraught problem of searching over all possible partitions.
The downside to our classification process is that it necessitates
a protracted calculation in serial, a fact that we sidestep on our 
examples by parallelizing our algorithm. Future work could seek
to streamline the computational effort to better address larger problems 
with more characteristics. Applying the classification process to
additional problems should also be fruitful.

\section*{Acknowledgements}
The authors are grateful for thoughtful comments made by Matthias Ehrgott.


\end{document}